\documentclass[a4paper,12pt]{article}
\usepackage[text={6.5in,8.4in},centering]{geometry}
\usepackage{graphicx,url,subfig}
\RequirePackage[OT1]{fontenc}
\RequirePackage{amsthm,amsmath,amssymb,amscd}
\RequirePackage[numbers]{natbib}
\RequirePackage[colorlinks,citecolor=blue,urlcolor=blue]{hyperref}
\usepackage{enumerate}
\usepackage{datetime}
\usepackage{etex}
\DeclareMathOperator{\sign}{sign}

\makeatother
\numberwithin{equation}{section}
\allowdisplaybreaks

\theoremstyle{plain}
\newtheorem{theorem}{Theorem}[section]
\newtheorem{corollary}[theorem]{Corollary}
\newtheorem{lemma}[theorem]{Lemma}

\theoremstyle{definition}

\newtheorem{remark}[theorem]{Remark}

\newtheorem{example}[theorem]{Example}

\newcommand{\B}{\textbf{B}}
\newcommand{\E}{\mathbb{E}}
\newcommand{\D}{\mathbb{D}}

\newcommand{\ud}{\ensuremath{\mathrm{d}}}

\newcommand{\Indt}[1]{1_{\left\{#1 \right\}}}

\newcommand{\Norm}[1]{\left|\left|  #1   \right|\right|}

\newcommand{\InPrd}[1]{\left\langle #1 \right\rangle}

\newcommand{\calF}{\mathcal{F}}

\newcommand{\calK}{\mathcal{K}}

\newcommand{\calI}{\mathcal{I}}
\newcommand{\calL}{\mathcal{L}}
\newcommand{\calM}{\mathcal{M}}

\newcommand{\R}{\mathbb{R}}

\newcommand{\Erf}{\ensuremath{\mathrm{erf}}}
\newcommand{\Erfc}{\ensuremath{\mathrm{erfc}}}

\DeclareMathOperator{\Lip}{\mathit{L}}

\DeclareMathOperator{\lip}{\mathit{l}}

\title{Two-point correlation function and Feynman-Kac formula for the stochastic heat equation}
\author{
{\bf Le Chen\footnote{
Research partially supported by a fellowship from Swiss National Science Foundation (P2ELP2\_151796),
\url{chenle02@gmail.com} or \url{chenle@ku.edu}.}}
\, \,   {\bf Yaozhong Hu}\footnote{Research
partially supported by a grant from the Simons Foundation
\#209206, \url{yhu@ku.edu}.}
\, \,   {\bf David Nualart}\footnote{Research partially supported by the NSF grant  DMS1512891 and the ARO grant FED0070445,
\url{nualart@ku.edu}.}
\\[1em]
Department of Mathematics\\
University of Kansas\\
Lawrence, Kansas, 66045, USA
\date{}
}

\begin{document}
\maketitle
\begin{center}
\begin{minipage}[rct]{5 in}
\footnotesize \textbf{Abstract:}
In this paper, we obtain an explicit formula for the two-point correlation function
for the solutions to the stochastic heat equation on $\R$.
The bounds for $p$-th moments proved in \cite{ChenDalang13Heat} are simplified.
We validate the Feynman-Kac formula for the $p$-point correlation function
of the solutions to this equation with measure-valued initial data.

\vspace{2ex}

\vspace{2ex}
\vspace{4ex}
\end{minipage}
\end{center}

% \tableofcontents
\setlength{\parindent}{1.5em}

%%%%%%%%%%%%%%%%%%%%%%%%%%%%%%%%%%%%%%%%%%%%%%
%%%%% MAIN: The chapters of the thesis
%%%%%%%%%%%%%%%%%%%%%%%%%%%%%%%%%%%%%%%%%%%%%%

% \mainmatter
% \graphicspath{{../figs/}}

\section{Introduction}
Consider the following stochastic heat equation
\begin{align}\label{EH:Heat}
\begin{cases}
\displaystyle
\left(\frac{\partial }{\partial t} - \frac{\nu}{2}
\frac{\partial^2 }{\partial x^2}\right) u(t,x) =  \rho(u(t,x))
\:\dot{W}(t,x),&
x\in \R,\; t >0, \\
% \displaystyle
\quad u(0,\cdot) = \mu(\cdot)\;,
\end{cases}
\end{align}
where $\dot{W}$ is a space-time white noise and $\rho: \R \to \R$ is a
globally Lipschitz function. The initial data $\mu$ is a signed Borel measure,
which we assume belongs to the set
\[
   \calM_H(\R) := \left\{\text{signed Borel measures $\mu$, s.t.}\:
\int_\R e^{-a x^2} |\mu|(\ud x)<+\infty,\: \text{for all $a>0$}\right\}.
\]
In the above, we denote $|\mu|:= \mu_+ + \mu_-$,  where $\mu=\mu_+-\mu_-$ and
$\mu_\pm$ are the two non-negative Borel measures with disjoint support that
provide the Jordan decomposition of $\mu$.
The set $\calM_H(\R)$ can be equivalently characterized by the condition that
\begin{align}\label{EH:J0finite}
\left(|\mu| * G_\nu(t,\cdot)\right)
(x) = \int_\R G_\nu(t,x-y) |\mu|(\ud y)<+\infty\;, \quad \text{for all $t>0$
and $x\in\R$},
\end{align}
where $*$ denotes the convolution in the space variable and $G_\nu(t,x)$ is
the one-dimensional heat kernel function
\begin{align*}
% \label{E2:1G-Heat}
G_\nu(t,x) := \frac{1}{\sqrt{2\pi \nu t}} \exp\left\{-\frac{x^2}{2\nu
t}\right\},\quad t>0,\: x\in\R\:.
\end{align*}
The initial condition $u(0, \cdot)=\mu(\cdot) $ is understood 
as $\lim_{t\downarrow 0}  u(t,\cdot)=\mu(\cdot)$  in the sense of distribution
(we identify a measure as a   distribution in the usual sense; see \cite[Theorem 1.7]{ChenKim14Comparison}).

Denote
\[
J_0(t,x):= (\mu*G_\nu(t,\cdot))(x)=\int_\R G_\nu(t,x-y)\mu(\ud y).
\]
Define the kernel function
\begin{align}
\label{E:K}
 \calK(t,x)=\calK(t,x;\nu,\lambda)&:=G_{\frac{\nu}{2}}(t,x)\cdot 
\left(\frac{\lambda^2}{\sqrt{4\pi\nu t}}+\frac{\lambda^4}{2\nu}
\: e^{\frac{\lambda^4 t}{4\nu}}\Phi\left(\lambda^2
\sqrt{\frac{t}{2\nu}}\right)\right),
\end{align}
where $\Phi(x)=\int_{-\infty}^x (2\pi)^{-1/2}e^{-y^2/2}\ud y$.
Some functions related to $\Phi(x)$ are the error functions $\Erf(x)=
\frac{2}{\sqrt{\pi}}\int_0^x e^{-y^2}\ud y$ and $\Erfc(x)=1-\Erf(x)$. Note that
$\Phi(x)=\Erfc(-x/\sqrt{2})/2$.

When $\rho(u)=\lambda u$, the following moment formula is proved in \cite{ChenDalang13Heat}
\begin{align}\label{E:SecMom}
\E \left(u(t,x)^2\right)  = J_0^2(t,x) +
 (J_0^2\star \calK)(t,x),
\end{align}
where 
%$\Norm{\cdot}_p$ denotes the $L^p(\Omega)$-norm, and 
``$\star$'' denotes the convolution in both space and time
variables, that is,
\begin{equation}
 (J_0^2\star \calK)(t,x):=
 \int_0^t\ud s \int_\R \ud y
 \: J_0^2(s,y)\calK(t-s,x-y).
\end{equation}
As for the two-point correlation function, define
\begin{align*}
\calI(t,x_1,\tau, x_2;\nu,\lambda) :=&
\lambda^2\int_0^t \ud r\int_\R  \ud z
\left[
J_0^2(r,z)+\left(J_0^2(\cdot,\circ)\star
\calK(\cdot,\circ;\nu,\lambda)\right)(r,z)
\right]\\
& \qquad\qquad \times G_\nu(t-r,x_1-z)G_\nu(\tau-r,x_2-z).
\end{align*}
Then by \cite[(2.26)]{ChenDalang13Heat},
for all $\tau\ge t>0$ and $x_1, x_2\in\R$,
\begin{equation}
\label{E:TP}
 \E\left[u(t,x_1)u\left(\tau,x_2\right)\right] =
J_0(t,x_1)J_0\left(\tau,x_2\right)  + \calI(t,x_1,\tau,x_2;\nu,\lambda)\;.
\end{equation}

The first goal of this note is to simplify the moment formulas \eqref{E:SecMom} and \eqref{E:TP} for the case $t=\tau$.
Note that the terms $(J_0^2\star \calK)(t,x)$ and $\calI(\dots)$ involve  four and six integrals, respectively.
We will reduce these integrals into only two integrals:
two convolutions of the initial data with respect to a kernel function.

% 
% When the initial data is function-valued (instead of a measure),
% there is a famous representation for the moments of \eqref{EH:Heat}
% --- the Feynman-Kac formulas (see \cite{HuNualart09}).
% Let $B_t^1$ and $B_t^2$ be two independent standard Brownian motions.
% Suppose that $\mu(\ud x)=u_0(x)\ud x$ where $u_0$ is a bounded measurable function.
% Then
% \begin{equation}
% \label{E:FK}
% \begin{aligned}
% \E\left(u(t,x_1)u(t,x_2)\right)=
% \E\Bigg[ &u_0\left(x_1+B_{\nu t}^1\right)u_0\left(x_2+B_{\nu t}^2\right)
% \\
% &\times \exp\left(\lambda^2\int_0^t\delta_{x_2-x_1}\left(B_{\nu s}^1-B_{\nu s}^2\right)\ud s\right)\Bigg],
% \end{aligned}
% \end{equation}
% where the expectation is with respect to the two Brownian motions.
% Evaluation of this expectation is not easy since it involves the joint law of a standard Brownian $B_t$ and
% its local time $L_t^a$ at an arbitrary level $a\in\R$.
% To the best of our knowledge, we are not aware of any references for this joint law expect the case where $a=0$.
% We will derive this joint distribution and then
% give an alternative and more probabilistic proof of the formula for the two-point correlation function.

\bigskip
It is well known that if the initial data is a function, then
the moments of the solution to \eqref{EH:Heat} with $\rho(u)=\lambda u$ admit
a Feynman-Kac representation; see \cite{HuNualart09}.
Suppose that $\mu(\ud x)=u_0(x)\ud x$ where $u_0$ is a bounded measurable function.
This representation says that for all $x_i\in\R$, $i=1,\dots,n$,
\begin{align}\label{E:FKF}
 \E\left[\prod_{i=1}^n u(t,x_i)\right]
 =\E^{\B}\left[\prod_{i=1}^n u_0\left(x_i+B_t^{i}\right)\:
 \exp\left(\lambda^2 \sum_{1\le i<j\le n}\int_0^t\delta_{x_j-x_i}
 \left(B_s^i-B_s^j\right)\ud s\right)\right],
\end{align}
where $\{B^i_t,t\ge 0\}$, $i=1,\dots, n$, are i.i.d. standard Brownian motions on $\R$, and
the expectation is with respect to all these Brownian motions.

On  one hand, direct evaluation of this expectation when $n=2$ is not easy
since it involves the joint law of a standard Brownian motion $B_t$ and
its local time $L_t^a$ at an arbitrary level $a\in\R$.
To the best of our knowledge, we are not aware of any references for this joint law expect the case where $a=0$.
We will derive this joint distribution and then
give an alternative and more probabilistic proof of the formula for the two-point correlation function.

On the other hand,
when the initial condition is a measure the meaning of \eqref{E:FKF}   is not clear.
Another aim of this paper is to make sense of \eqref{E:FKF} for initial data in $\calM_H(\R)$ using Malliavin calculus. More precisely, we show that  $\prod_{i=1}^n u_0\left(x_i+B_t^{i}\right)$   belongs to the Meyer-Watanabe space   $\mathbb{D}^{-\alpha, p}$ of Wiener distributions  for any $p>1$ and $\alpha > n(1-1/p)$, and  the exponential factor in  (\ref{E:FKF})
\begin{equation}  \label{eq1}
Y_n:= \exp\left(\lambda^2 \sum_{1\le i<j\le n}\int_0^t\delta_{x_j-x_i}
 \left(B_s^i-B_s^j\right)\ud s\right)
 \end{equation}
  belongs to $\mathbb{D}^{\alpha, p}$  for any $p>1$ and $\alpha <\frac 12$.  Then, we can choosing $p$ such that $n(1-1/p)< \alpha <1/2$, we can write
\begin{equation}  \label{eq2}
  \E\left[\prod_{i=1}^n u(t,x_i)\right] =\left\langle \prod_{i=1}^n u_0\left(x_i+B_t^{i}\right), Y_n \right\rangle,
 \end{equation}
  where $\langle  \cdot,  \cdot\rangle$ denotes the duality between  $\mathbb{D}^{-\alpha, p}$ and $ \mathbb{D}^{\alpha, q}$,  if  $1/p+1/q=1$. 

\medskip
When $\rho(u)$ in \eqref{EH:Heat} is nonlinear but satisfies the global Lipschitz 
condition,  the explicit  formula for the moment  of the solution is impossible.
We obtain  an  upper bound for the $p$-moment  of the solution and a lower 
bound for the second moment.  The idea is to compare them with the ones in the  linear   case.

\bigskip
We first state our main results in Section \ref{Sec:Main}.
The result for the two-point correlation function, Theorem \ref{T:SecMom}, is proved in Section \ref{Sec:MainProof}.
In Section \ref{Sec:LocB}, the joint law of $(B_t,L_t^a)$, Theorem \ref{T:LocB}, is proved.
In Section \ref{Sec:FeynmanKac}, we give the alternative proof of our two-point correlation formula for function-valued initial data.
% as a consequence we prove Corollary \ref{C:ExpL}.
Finally, by proving Theorems \ref{T:mun} and \ref{T:fL} in Sections \ref{SS:mun} and \ref{SS:fL},
respectively, we make sense of \eqref{E:FKF} for measure-valued initial data.

\bigskip
Throughout of the paper, $\Norm{\cdot}_p$ denotes $L^p(\Omega)$ norm.

\section{Main Results}\label{Sec:Main}
\subsection{Formulas for two-point correlation function}
\begin{theorem}\label{T:SecMom}
Suppose that $\mu\in\calM_H(\R)$. If $\rho(u)=\lambda u$, then for all $t>0$ and $x_1,x_2\in\R$,
\begin{align}\label{E:SecTP}
\E\left[u(t,x_1)u(t,x_2)\right] &=
\iint_{\R^2}\mu(\ud z_1)\mu(\ud z_2)
\:\calK^*(t,x_1-z_1,x_2-z_2,x_1-x_2;\lambda),
\end{align}
or
\begin{align}\notag
 \E\left[u(t,x_1)u(t,x_2)\right] =& \quad J_0(t,x_1)J_0(t,x_2)\\
&+
\iint_{\R^2}\mu(\ud z_1)\mu(\ud z_2)
\:\calK^\dagger(t,x_1-z_1,x_2-z_2,x_1-x_2;\lambda),
\label{E:SecTP2}
\end{align}
where
\begin{equation}
\begin{aligned}
 \calK^\dagger(t,z_1,z_2,y;\lambda)=&
 \frac{\lambda^2}{2\nu} G_{\nu/2}\left(t,\frac{z_1+z_2}{2}\right)
 \exp\left(\frac{\lambda^2}{4\nu}\left[
\lambda^2 t-2(|y|+|y-(z_1-z_2)|)
\right]\right)\\ 
&\times
\Phi\left( \frac{\lambda^2 t -(|y|+|y-(z_1-z_2)|)}{\sqrt{2\nu t}}\right).
\end{aligned}
\end{equation}
and
\begin{align} \notag
\calK^*(t,z_1,z_2,y;\lambda) =&
 G_{\nu}(t,z_1)G_{\nu}(t,z_2) + \calK^\dagger(t,z_1,z_2,y;\lambda)\\ \notag
=&G_{\nu/2}\left(t,\frac{z_1+z_2}{2}\right)\Bigg[
G_{2\nu}(t,z_1-z_2) \\ \label{E:calKStar}
&+ \frac{\lambda^2}{2\nu} \exp\left(\frac{\lambda^2}{4\nu}\left[
\lambda^2 t-2(|y|+|y-(z_1-z_2)|)
\right]\right)\\ \notag
&\times
\Phi\left( \frac{\lambda^2 t -(|y|+|y-(z_1-z_2)|)}{\sqrt{2\nu t}}\right)
\Bigg].
\end{align}
In particular,
\begin{align}\label{E:Sec}
\Norm{u(t,x)}_2^2 &=
\iint_{\R^2}\mu(\ud z_1)\mu(\ud z_2)
\:\calK^*(t,x-z_1,x-z_2,0;\lambda).
\end{align}
\end{theorem}

In the following, we will use the convention that for any pair $(w_1,w_2)$ of variables, 
\begin{align}\label{E:barDelta}
\bar{w}:=(w_1+w_2)/2\quad\text{and}\quad
\Delta w:=w_2-w_1.
\end{align}

\begin{remark}
Formulas \eqref{E:SecTP} and \eqref{E:SecTP2} are in the convolution form.
One can also write them
in the following inner product form:
\begin{align} \label{E:Sec_Inner}
 \E\left[u(t,x_1)u(t,x_2)\right] &=
 \iint_{\R^2} \mu(\ud z_1)\mu(\ud z_2) \: K^* \left(t,x_1,x_2,z_1,z_2;\lambda\right),
\end{align}
or
\begin{align} \label{E:Sec_Inner2}
 \E\left[u(t,x_1)u(t,x_2)\right] &=
 J_0(t,x_1)J_0(t,x_2)+ \iint_{\R^2} \mu(\ud z_1)\mu(\ud z_2) \: K^\dagger \left(t,x_1,x_2,z_1,z_2;\lambda\right),
\end{align}
where
\begin{equation}
\begin{aligned}
 K^\dagger \left(t,x_1,x_2,z_1,z_2;\lambda\right)= & 
 \frac{\lambda^2}{2\nu}G_{\nu/2}(t,\bar{x}-\bar{z}) \exp\left(-\frac{\lambda^2(|\Delta x|+|\Delta z|)}{2\nu}+\frac{\lambda^4t}{4\nu}\right)\\
&\times \Phi\left(\frac{\lambda^2\sqrt{t}}{\sqrt{2\nu}}-\frac{|\Delta x|+|\Delta z|}{\sqrt{2\nu t}}\right),
\end{aligned}
\end{equation}
and
\begin{align} \notag
 K^* \left(t,x_1,x_2,z_1,z_2\right)=& G_{\nu}(t,x_1-z_1)G_{\nu}(t,x_2-z_2) + K^\dagger \left(t,x_1,x_2,z_1,z_2;\lambda\right)\\ \notag
 =&G_{\nu/2}(t,\bar{x}-\bar{z}) \Bigg[G_{2\nu}(t,\Delta x- \Delta z)+
 \frac{\lambda^2}{2\nu} \exp\left(-\frac{\lambda^2(|\Delta x|+|\Delta z|)}{2\nu}+\frac{\lambda^4t}{4\nu}\right)\\
&\times \Phi\left(\frac{\lambda^2\sqrt{t}}{\sqrt{2\nu}}-\frac{|\Delta x|+|\Delta z|}{\sqrt{2\nu t}}\right)\Bigg].
\label{E:K*}
\end{align}
\end{remark}

\begin{example}[Delta initial data]
When $\mu=\delta_0$, then
\[
\Norm{u(t,x)}_2^2 = \calK^*(t,x,x,0;\lambda) = \lambda^{-2}\: \calK(t,x),
\]
and
\begin{align*}
\E\left[u(t,x_1)u(t,x_2)\right] =& \calK^*(t,x_1,x_2,x_1-x_2;\lambda)
\\
=& G_{\nu}(t,x_1)G_\nu(t,x_2)+ \calK^\dagger(t,x_1,x_2,x_1-x_2;\lambda)\\
=&G_{\nu}(t,x_1)G_\nu(t,x_2)
+ \frac{\lambda^2}{2\nu}  G_{\frac{\nu}{2}}\left(t,\frac{x_1+x_2}{2}\right)\\
&\times
\exp\left(\frac{
\lambda^2(\lambda^2 t-2|x_1-x_2|)}{4\nu}\right)
\Phi\left( \frac{\lambda^2 t -|x_1-x_2|}{\sqrt{2\nu t}}\right).
\end{align*}
This recovers the results in \cite[Corollary 2.8]{ChenDalang13Heat}.
\end{example}

\begin{remark}
Note that the function $m_2(t,x_1,x_2)=K^*(t,x_1,x_2,0,0)$,
or equivalently $m_2(t,x_1,x_2)=\calK^*(t,x_1,x_2,x_1-x_2)$,  solves the
following parabolic equation (see e.g., \cite[Theorem 3.2 on p. 46]{CarmonaMolchanov94PAM})
\[
\begin{cases}
\displaystyle
\frac{\partial }{\partial t} m_2(t,x_1,x_2)=H_2(\nu,\lambda)\: m_2(t,x_1,x_2), & \quad t>0, x_1, x_2\in\R,\\[1em]
\displaystyle
m_2(0,x_1,x_2) = \delta_0(x_1)\delta_0(x_2),
\end{cases}
\]
where the operator
\[
H_2(\nu,\lambda)=\frac{\nu}{2}\left(\frac{\partial^2}{\partial x_1^2}+\frac{\partial^2}{\partial x_2^2}\right)+ \lambda \delta_0(x_1-x_2)
\]
is the {\it $2$-particle Schr\"odinger operator} (see \cite{CarmonaMolchanov94PAM}).
\end{remark}

\begin{example}[Lebesgue's initial measure]
When $\mu(\ud x)=\ud x$, then from \eqref{E:SecTP2} or \eqref{E:Sec_Inner2}, 
\begin{align}\notag
 \E\left[u(t,x_1)u(t,x_2)\right]&= 1 + \iint_{\R^2} \ud z_1\ud z_2 \: \calK^\dagger(t,z_1,z_2,x_1-x_2;\lambda)\\ \notag
 &=1+ \frac{\lambda^2}{2\nu} \int_{\R} \ud z
 \exp\left(\frac{\lambda^2}{4\nu}\left[
\lambda^2 t-2(|x_1-x_2|+|x_1-x_2-z|)
\right]\right)\\ \notag
&\qquad\times
\Phi\left( \frac{\lambda^2 t -(|x_1-x_2|+|x_1-x_2-z)|)}{\sqrt{2\nu t}}\right)\\
&=2 e^{\frac{\lambda ^4 t-2 \lambda ^2 |x_1-x_2|}{4 \nu }} \Phi\left(\frac{\lambda
   ^2 t-|x_1-x_2|}{\sqrt{2 \nu t}}\right)+2\Phi\left(\frac{|x_1-x_2|}{\sqrt{2\nu t}}\right)-1,
   \label{E:SecTPConst}
\end{align}
and in particular,
\begin{align}\label{E:SecConst}
 \Norm{u(t,x)}_2^2 &=
 2 e^{\frac{\lambda ^4 t}{4 \nu }} \Phi\left(\frac{\lambda
   ^2 \sqrt{t}}{\sqrt{2 \nu}}\right).
\end{align}
These two formulas \eqref{E:SecTPConst} and \eqref{E:SecConst} recover the results in \cite[Corollary 2.5]{ChenDalang13Heat}.
Note that the equality in \eqref{E:SecTPConst} can be established through integration by parts.
\end{example}

\begin{corollary}\label{C:ExpL}
Let $L^x_t$ be the local time of the standard Brownian motion.
Then for all $\lambda\in\R$, $t>0$, and $x\in\R$,
\[
\E\left[\exp\left(\lambda^2 L^x_t\right)\right] =
2\: e^{\lambda^4 t/2-\lambda ^2 |x|} \Phi\left(\lambda^2 \sqrt{t}-|x|/\sqrt{t}\right)
+2\Phi\left(|x|/\sqrt{t}\right)-1.
\]
In particular,
\[
\E\left[\exp\left(\lambda^2 L^x_t \right)\right] \le 2\:e^{\lambda^4 t/2}+1.
\]
\end{corollary}

\begin{proof}
Let $u(t,x)$ be a solution to \eqref{EH:Heat} with $\rho(u)=\sqrt{2}\: \lambda u$, $\nu=1$, and $u_0(x)\equiv 1$.
By the Feynman-Kac formula \eqref{E:FKF} with $n=2$, $x_1=0$, and $x_2=x$,
\begin{align*}
\E[u(t/2,0)u(t/2,x)] &= \E \Big[\exp\Big(2\lambda^2 \int_0^{t/2} \delta_{x}(B_s^1-B_s^2)\ud s \Big)\Big]
= \E\Big[\exp\Big(2\lambda^2 \int_0^{t/2} \delta_{x}(B_{2s}')\ud s \Big)\Big]\\
&=\E\Big[\exp\Big(\lambda^2 \int_0^{t} \delta_{x}(B_{r}')\ud r \Big)\Big]=
\E\left[\exp\left(\lambda^2 L^x_t\right)\right],
\end{align*}
where $B'_t$ is a standard Brownian motion.
Then apply \eqref{E:SecTPConst} with $\lambda^2$, $t$, and $\nu$ replaced by
$2\lambda^2$, $t/2$, and $1$, respectively.
\end{proof}
% The proof of this corollary is given at the end of Section \ref{Sec:FeynmanKac}.

\bigskip
For the $p$-th moments, we have the following bound, which simplifies the expression in \cite{ChenDalang13Heat}.
Denote 
\[
c_p=\begin{cases}
     1& \text{if $p=2$},\cr
     2& \text{if $p>2$}.
    \end{cases}
\]
\begin{theorem}
Let $u(t,x)$ be a solution to \eqref{EH:Heat} starting from $\mu\in\calM_H(\R)$.
For all $t>0$ and $x\in\R$ , the following moment bounds hold:
\begin{enumerate}[(1)]
 \item If $|\rho(x)| \le \Lip_\rho|x|$ for all $x\in\R$, then for all $p\ge 2$,
 \[
\Norm{u(t,x)}_p^2 \le \iint_{\R^2} |\mu|(\ud z_1)|\mu|(\ud z_2) \overline{\calK}^*_{p,\Lip_\rho}(t,x-z_1,x-z_2),
 \]
where
\[
 \overline{\calK}^*_{p,\Lip_\rho}(t,z_1,z_2) =c_p\: \calK^*(t,z_1,z_2,0;c_p^2 \sqrt{p/2}\Lip_\rho);
\]
 \item If $|\rho(x)| \ge \lip_\rho|x|$ for all $x\in\R$ and if $\mu\ge 0$, then
 \[
\Norm{u(t,x)}_2^2 \ge \iint_{\R^2} \mu(\ud z_1)\mu(\ud z_2) \calK^*(t,x-z_1,x-z_2,0;\lip_\rho).
 \]
%  where
%  \[
%  \underline{\calK}^*_{\lip_\rho}(t,x-z_1,x-z_2) =
%  \]
\end{enumerate}
\end{theorem}
\begin{proof}
Part (2) is clear. As for (1), note that the upper bounds for both $2$nd and $p$-th moments
share similar forms (see \cite[(2.21)]{ChenDalang13Heat}), but with different parameters.
By the notation in \cite{ChenDalang13Heat}, $a_{p,0}=\sqrt{c_p}$ and $z_p=c_p^{3/2}\sqrt{p/2}$.
Then   replacing the parameter $\lambda$ in $\calK^*$ by $a_{p,0}z_p\Lip_\rho$ and multiplying it by a factor $c_p$,
one passes from $2$nd moment to $p$-th moment (see \cite[Theorem 2.4]{ChenDalang13Heat}).
% Finally, note that when $p=2$ and $\rho(u)=\lambda u$, the 
\end{proof}

\bigskip
For the alternative proof of Theorem \ref{T:SecMom},
we will need the following joint density of the standard Brownian motion $B_t$ and its local
time $L^a_t $ at a level $a\in\R$, which is by itself interesting.
When $a=0$, it is well known (see, e.g., \cite[Exercise 1, on p. 181]{ChungWilliams90}) that this law is
\begin{align}\label{E:LocB0}
P\left(B_t^1\in\ud y, L_t^0\in\ud v\right)
=\frac{|y|+v}{\sqrt{2\pi t^3}} \exp\left(-\frac{\left(|y|+v\right)^2}{2t}\right) %\Indt{y\in\R} 
\Indt{v\ge 0}\ud y\ud v.
\end{align}
% To the best of our knowledge, we are not aware of any references for the joint law of $(B_t,L_t^a)$ with $a\ne 0$,
% which is proved in the following theorem.
More generally, we have the following theorem.

\begin{theorem}\label{T:LocB}
The joint distribution of $(B_t,L_t^a)$ for $t>0$ and $a\in\R$ is
\begin{equation}\label{E:LocB}
\begin{aligned}
P\left(B_t^1\in\ud y, L_t^a\in\ud v\right) = \Indt{v\ge 0}&
\Bigg[\:\frac{|a|+|y-a|+v}{(2\pi t)^3} \exp\left(-\frac{(|a|+|y-a|+v)^2}{2t}\right)\\
& + \frac{1}{\sqrt{2\pi t}}\left(e^{-\frac{y^2}{2t}}-e^{-\frac{(2a-y)^2}{2t}}\right)\Indt{\sign(a)y\le |a|} \delta_{0}(v)\Bigg] \ud y\ud v,
\end{aligned}
\end{equation}
where $\sign(a)=1$ if $a\ge 0$ and $-1$ if $a<0$.
\end{theorem}

% The law in \eqref{E:LocB} reduces to \eqref{E:LocB0} when $a=0$.

\begin{corollary}
The law of $L_t^a$ for $t>0$ and $a\in\R$ is
\[
P\left(L_t^a\in\ud v\right) = \left(\frac{\sqrt{2}}{\sqrt{\pi t}} \exp\left(-\frac{(v+|a|)^2}{2 t}\right)+\left[2\Phi\left(\frac{|a|}{\sqrt{t}}\right)-1\right]\delta_0(v)\right)
\Indt{v\ge 0}\ud v.
\]
\end{corollary}
\begin{proof}
Integrate the right hand of \eqref{E:LocB} for $\ud y$ over $\R$. For the first term, use the integration-by-parts formula. 
For the second term, use the definition of $\Phi(\cdot)$.
\end{proof}

\subsection{Feynman-Kac formulas for measure-valued initial data}

% ?? Say something about the spaces and Malliavin calculus ...

Let $H$ be a Hilbert space with inner product $\InPrd{\cdot,\cdot}$ and
let $W=\{W(h), \: h\in H\}$ be a zero mean Gaussian process
with covariance function $\E(W(h)W(g))=\InPrd{h,g}$.
For any square integrable random variable $F\in L^2(\Omega)$, let 
\[
F=\E[F] +\sum_{n=1}^\infty I_n(f_n)
\]
be its chaos expansion, where $f_n\in H^{\hat{\otimes}n}$ (symmetric tensor product)
and $I_n$ denotes the multiple stochastic integral. 
Let $L$ be the generator of the {\it Ornstein-Uhlenbeck} semigroup, i.e., $LF=-\sum_{n=1}^\infty n I_n(f_n)$.
For any $s\in\R$, denote by $\D^{s,p}(H)$ the completion of $H$-valued polynomial random variables with respect to the norm
\[
\Norm{F}_{s,p}=\Norm{(I-L)^{s/2} F}_{L^p(\Omega;H)}.
\]
We refer \cite{Nualart06} for more details.

 In our framework, $H=L^2(\mathbb{R}_+;\mathbb{R}^n)$ and for any $h\in H$, $W(h)= \sum_{i=1}^n \int_0^\infty h^i_s dB^i_s$, where $\{B_t^i, t\ge 0\}$, $i=1\dots, n$, are $n$ independent standard Brownian motions on $\mathbb{R}$. 

\begin{theorem}\label{T:mun}
For any $\mu_i\in\calM_H(\R)$, $x_i\in\R$, $h_i\in H$, $i=1,\dots,n$. It holds that
\[
\prod_{i=1}^n\mu_i(W (h_i)+x_i)\in \D^{-\alpha,p}(\R)\qquad \text{if $\alpha+n/p>n$, $\alpha>0$ and $p>1$.}
\]
\end{theorem}

\begin{theorem}\label{T:fL}
Let $ \widetilde{B}^i =\{\widetilde{B}_t^i ,t\ge 0\}$, $i \in I $, be one-dimensional Brownian motions  belonging to Gaussian space spanned by $W$, where $I$ is a finite set with $m$ elements.
Let $L_t^{i,x} $, $i \in I$, be the local time of $\widetilde{B}^i$.
Suppose that the function $f:\R^m_+\mapsto \R$  satisfies that 
$\frac{\partial^2   }{\partial x_i^2} f \ge 0$ and
\begin{align}\label{E:Ct}
% C_{t,x,p}:=
\max_{i=1,\dots,m}
% \mathop{
\sup_{\epsilon_1,\dots,\epsilon_m\in (-1,1)}
% }_{j=1,\dots,n}
\Norm{f_i\left(L^{1,x_1+\epsilon_1}_t +\epsilon_1,\dots,L^{m,x_m+\epsilon_m}_t+\epsilon_m\right)}_p<+\infty,
\end{align}
for all $t\ge 0$, $x_i\in\R$, and $p\ge 2$, where $f_i=\frac{\partial}{\partial x_i}f$.
Then
\[
f\left(L^{1,x_1}_t,\dots, L^{m,x_m}_t\right)\in \D^{\alpha,p}(\R) \qquad \text{if $p> 1$ and $\alpha<1/2$.}
\]
In particular, one can choose $f(z_1,\dots,z_m)=\exp\left(\lambda^2 \sum_{j=1}^m z_j\right)$, $\lambda\in\R$.
\end{theorem}

Choosing  $h_i =\mathbf{1}_{[0,t]}$, $1\le i  \le  n$, Theorem  \ref{T:mun}  implies that
\begin{align}\label{E:u}
\prod_{i=1}^n u_0\left(x_i+B_t^{i}\right)\in \D^{-\alpha,p}(\R) \qquad \text{if $\alpha+n/p>n$, $\alpha>0$ and $p>1$,}
\end{align}
for $u_0\in\calM_H(\R)$. On the other hand,  choosing $I=\{(j,k): 1\le  j<k \le n\}$ and  $\widetilde{B}^{(j,k)} =B^j-B^k$, Theorem \ref{T:fL} implies that the random variable $Y_n$ defined in  (\ref{eq1})  belongs to $  \D^{\alpha,q}(\R) $ if $q> 1$ and $\alpha<1/2$.
Therefore, by choosing $p$ close to $1$ such that $n(1-1/p)<1/2$, and choosing $q$ such that $1/p+1/q=1$,
one  can see that the Feynman-Kac formula \eqref{E:FKF} is well defined  as  the  duality relationship (\ref{eq1}), for any $u_0\in\calM_H(\R)$. This can be proved   using the fact that   $u_0 *G_1(\epsilon, \cdot)$  converges to $u_0$,  as $\epsilon$ tends to zero,  in the topology of  $\D^{-\alpha,p}(\R) $.

\section{Proof of Theorem \ref{T:SecMom}}
\label{Sec:MainProof}
\begin{proof}[Proof of Theorem \ref{T:SecMom}]
The proof consists  of  the following two steps:

{\vspace{1em}\bf\noindent Step 1.~} We first prove the second moment \eqref{E:Sec}. Write $J_0^2$ in the form of double integrals
\begin{align*}
 (J_0^2\star \calK)(t,x) &=
 \int_0^t\ud s \int_\R \ud y
 \: \calK(t-s,x-y) \iint_{\R^2}\mu(\ud z_1)\mu(\ud z_2) G_\nu(s,y-z_1)G_\nu(s,y-z_2).
\end{align*}
By \cite[Lemma 5.4]{ChenDalang13Heat} and the notation \eqref{E:barDelta},
\begin{align}\label{E:GG}
 G_\nu(s,y-z_1)G_\nu(s,y-z_2) = G_{\nu/2}(s,y-\bar{z}) G_{\nu}(2s, \Delta z).
\end{align}
Denote
\begin{align}
 \label{E2:H}
H(t)=\frac{1}{\sqrt{4\pi\nu t}}+\frac{\lambda^2}{2\nu}
\: e^{\frac{\lambda^4 t}{4\nu}}\Phi\left(\lambda^2
\sqrt{\frac{t}{2\nu}}\right).
\end{align}
By Fubini's theorem and the semigroup property of the heat kernel,
\begin{align*}
 (J_0^2\star \calK)(t,x) =&
 \lambda^2 \iint_{\R^2}\mu(\ud z_1)\mu(\ud z_2)
 \int_0^t\ud s\: G_\nu(2s,\Delta z) \int_\R \ud y
 \: \calK(t-s,x-y) G_{\nu/2}(s,y-\bar{z})\\
 =&
 \lambda^2\iint_{\R^2}\mu(\ud z_1)\mu(\ud z_2)
 G_{\nu/2}(t,x-\bar{z}) \int_0^t\ud s\: G_\nu(2s,\Delta z) H(t-s).
\end{align*}
Now we will use the Laplace transform to evaluate the $\ud s$-integral.
By \cite[(27) on p. 146]{Erdelyi1954-I},
\begin{align}\label{E:LapG}
\calL[G_{2\nu}(\cdot, x)](z) =\frac{\exp\left(-\nu^{-1/2} \left| x\right|
   \sqrt{z}\right)}{2 \sqrt{\nu  z}}.
\end{align}
By (1) on p. 137  and (5) on p. 176 of \cite{Erdelyi1954-I},
\begin{align*}
\calL[H](z)=&
\frac{\lambda^2}{4 \nu z-\lambda ^4}+\frac{1}{2 \sqrt{\nu
    z}}+\frac{\lambda^4}{2  \sqrt{\nu z} \left(4 \nu
   z-\lambda ^4\right)}\\
=&
\frac{1}{2}\left(\frac{1}{2\sqrt{\nu z}-\lambda^2}-\frac{1}{2\sqrt{\nu z}+\lambda^2}\right)
+\frac{1}{2\sqrt{\nu z}}\\
&+\frac{\lambda^2}{4\sqrt{\nu z}}\left(\frac{1}{2\sqrt{\nu z}-\lambda^2}-\frac{1}{2\sqrt{\nu z}+\lambda^2}\right).
\end{align*}
Hence,
\begin{align*}
 \calL[H](z)\calL[G_{2\nu}(\cdot, x)](z) &=
 f_{1,-}(z,x)-f_{1,+}(z,x) + f_{2}(z,x)
+ f_{3,-}(z,x)-f_{3,+}(z,x),
\end{align*}
where
\begin{align*}
f_{1,\pm}(z,x) =&\frac{\exp\left(-\nu^{-1/2} \left| x\right|
   \sqrt{z}\right)}{4\sqrt{\nu z}\left(2\sqrt{\nu z}\pm \lambda^2\right)},
\\
f_2(z,x) =&\frac{\exp\left(-\nu^{-1/2} \left| x\right|
   \sqrt{z}\right)}{4\nu z},
\\
f_{3,\pm}(z,x) =&\frac{\lambda^2\exp\left(-\nu^{-1/2} \left| x\right|
   \sqrt{z}\right)}{8\nu z\left(2\sqrt{\nu z}\pm \lambda^2\right)}.
\end{align*}
Apply \cite[(16) on p. 247]{Erdelyi1954-I} with $\alpha=\nu^{-1/2}|x|$ and $\beta=\pm\frac{\lambda^2}{\sqrt{4\nu}}$,
\begin{align}\label{E:f1}
\calL^{-1}[f_{1,\pm}(\cdot,x)](t)=&
\frac{1}{8\nu}e^{\pm\frac{\lambda^2 |x|}{2\nu}+\frac{\lambda^4 t}{4\nu}}
\Erfc\left(\frac{|x|}{\sqrt{4\nu t}}\pm\frac{\lambda^2\sqrt{t}}{\sqrt{4\nu}} \right)  .
\end{align}
Apply \cite[(3) on p. 245]{Erdelyi1954-I} with $\alpha=\nu^{-1}|x|^{2}$,
\begin{align*}
\calL^{-1}[f_{2}(\cdot,x)](t)=&
\frac{1}{4\nu}\Erfc\left(\frac{|x|}{\sqrt{4\nu t}}\right).
\end{align*}
Apply \cite[(14) on p. 246]{Erdelyi1954-I} with $\alpha=\nu^{-1/2}|x|$ and $\beta=\pm\frac{\lambda^2}{\sqrt{4\nu}}$,
\begin{align*}
 \calL^{-1}[f_{3,\pm}(\cdot,x)](t)=&
 \pm \frac{1}{8\nu}
 \Bigg(\Erfc\left(\frac{|x|}{\sqrt{4\nu t}}\right)-
 e^{\pm\frac{\lambda^2 |x|}{2\nu}+\frac{\lambda^4 t}{4\nu}}
\Erfc\left(\frac{|x|}{\sqrt{4\nu t}}\pm\frac{\lambda^2\sqrt{t}}{\sqrt{4\nu}} \right)\Bigg).
\end{align*}
Therefore, the $\ud s$-integral is equal to
\begin{align*}
\int_0^t\ud s\: G_{2\nu}(s,\Delta z) H(t-s)&=
 \frac{1}{4\nu}
 e^{-\frac{\lambda^2 |\Delta z|}{2\nu}+\frac{\lambda^4 t}{4\nu}}
\Erfc\left(\frac{|\Delta z|}{\sqrt{4\nu t}}-\frac{\lambda^2\sqrt{t}}{\sqrt{4\nu}} \right)\\
&=
 \frac{1}{2\nu}
 e^{-\frac{\lambda^2 |\Delta z|}{2\nu}+\frac{\lambda^4 t}{4\nu}}
 \Phi\left(\frac{\lambda^2\sqrt{t}}{\sqrt{2\nu}}-\frac{|\Delta z|}{\sqrt{2\nu t}}\right).
\end{align*}
Finally, by \eqref{E:GG},
\begin{align*}
 J_0^2(t,x) = \iint_{\R^2}\mu(\ud z_1)\mu(\ud z_1)
 G_{\nu/2}(t,x-\bar{z})
 G_{2\nu}(t,\Delta z).
\end{align*}
This proves the formula \eqref{E:Sec}, i.e.,
\begin{equation}\label{E_:Sec}
\begin{aligned}
\Norm{u(t,x)}_2^2=\iint_{\R^2}&\mu(\ud z_1)\mu(\ud z_2)\: G_{\nu/2}(t,x-\bar{z})\\
\times & \left(
G_{2\nu}(t,\Delta z) +
\frac{\lambda^2}{2\nu}
 e^{-\frac{\lambda^2 |\Delta z|}{2\nu}+\frac{\lambda^4 t}{4\nu}}
 \Phi\left(\frac{\lambda^2\sqrt{t}}{\sqrt{2\nu}}-\frac{|\Delta z|}{\sqrt{2\nu t}}\right)
\right).
\end{aligned}
\end{equation}

{\vspace{1em}\bf\noindent Step 2.~}
Now let us consider the two-point correlation function \eqref{E:SecTP}.
Fix $t>0$ and $x_1,x_2\in\R$.
Apply \eqref{E:TP}, \eqref{E:GG} and \eqref{E_:Sec}, and then integrate over $\ud y$ using the semigroup property:
\begin{align*}
 \calI(t,x_1,t,x_2;\nu,\lambda)=& \lambda^2 \int_0^t\ud r \int_\R \ud y\: \Norm{u(r,y)}_2^2 G_{\nu/2}(t-r,\bar{x}-y)G_{2\nu}(t-r,\Delta x)
 \\
 =&
 \lambda^2 \int_0^t\ud r \int_\R \ud y \: G_{\nu/2}(t-r,\bar{x}-y)G_{2\nu}(t-r,\Delta x) \iint_{\R^2}\mu(\ud z_1)\mu(\ud z_1)
 \\&\times G_{\nu/2}(r,y-\bar{z})
  \left( G_{2\nu}(r,\Delta z) +
\frac{\lambda^2}{2\nu}
 e^{-\frac{\lambda^2 |\Delta z|}{2\nu}+\frac{\lambda^4 r}{4\nu}}
 \Phi\left(\frac{\lambda^2\sqrt{r}}{\sqrt{2\nu}}-\frac{|\Delta z|}{\sqrt{2\nu r}}\right)
\right)
 \\
 =&\lambda^2\iint_{\R^2}\mu(\ud z_1)\mu(\ud z_2) \:G_{\nu/2}(t,\bar{x}-\bar{z})
 \int_0^t \ud r
 \: G_{2\nu}(t-r,\Delta x) \widetilde{H}(r,\Delta z),
\end{align*}
where
\begin{align*}
 \widetilde{H}(t,x)&:=
G_{2\nu}(t,x) + \frac{\lambda^2}{2\nu} e^{-\frac{\lambda^2|x|}{2\nu}+\frac{\lambda^4 t}{4\nu}}\:
\Phi\left(\lambda^2 \sqrt{\frac{t}{2\nu}}-\frac{|x|}{\sqrt{2\nu t}}\right). 
\end{align*}
Here, $\widetilde{H}(t,0)=H(t)$; see \eqref{E2:H}.
Notice that
\[
\frac{1}{4\nu} e^{-\frac{\lambda^2|x|}{2\nu}+\frac{\lambda^4 t}{4\nu}}\:
\Phi\left(\lambda^2 \sqrt{\frac{t}{2\nu}}-\frac{|x|}{\sqrt{2\nu t}}\right)
=
\frac{1}{8\nu} e^{-\frac{\lambda^2|x|}{2\nu}+\frac{\lambda^4 t}{4\nu}}\:
\Erfc\left(-\lambda^2 \sqrt{\frac{t}{4\nu}}+\frac{|x|}{\sqrt{4\nu t}}\right).
\]
Hence,
\[
\widetilde{H}(t,x)= G_{2\nu}(t,x) + 2\lambda^2 \calL^{-1}[f_{1,-}(\cdot,x)](t).
\]
% Apply \cite[(16), on p. 247]{Erdelyi1954-I} with $\alpha=|x|/\sqrt{\nu}$ and $\beta=-\lambda^2/\sqrt{4\nu}$,
% \begin{align*}
%  \calL[f](z)= \frac{\exp\left(-\nu^{-1/2}|x|\sqrt{z}\right)}{\sqrt{z}\left(\sqrt{z}-\lambda^2/\sqrt{4\nu}\;\right)}.
% \end{align*}
Together with \eqref{E:LapG}, we have that,
\begin{align*}
 \calL[\widetilde{H}(\cdot,x)](z)&=\frac{\exp\left(-\nu^{-1/2}|x|\sqrt{z}\right)}{\sqrt{4\nu z}-\lambda^2},
\end{align*}
and
\begin{align*}
 \calL[G_{2\nu}(\cdot,x)](z)
 \calL[\widetilde{H}(\cdot,x')](z)&=\frac{\exp\left(-\nu^{-1/2}(|x|+|x'|)\sqrt{z}\right)}{\sqrt{4\nu z}\left(\sqrt{4\nu z}-\lambda^2\right)}=2f_{1,-}(z,|x|+|x'|),
\end{align*}
By \eqref{E:f1}, 
\begin{align*}
 \int_0^t \ud r
 \: G_{2\nu}(t-r,\Delta x) \widetilde{H}(r,\Delta z)
 &= \frac{1}{2\nu}e^{-\frac{\lambda^2(|\Delta x|+|\Delta z|)}{2\nu}+\frac{\lambda^4t}{4\nu}}
\Phi\left(\frac{\lambda^2\sqrt{t}}{\sqrt{2\nu}}-\frac{|\Delta x|+|\Delta z|}{\sqrt{2\nu t}}\right).
\end{align*}
Therefore,
\begin{multline*}
 \E\left[u(t,x_1)u(t,x_2)\right] =J_0(t,x_1)J_0(t,x_2) + \iint_{\R^2}\mu(\ud z_1)\mu(\ud z_2) \: G_{\nu/2}(t,\bar{x}-\bar{z})
 \\
 \times \frac{\lambda^2}{2\nu}e^{-\frac{\lambda^2(|\Delta x|+|\Delta z|)}{2\nu}+\frac{\lambda^4t}{4\nu}}
\Phi\left(\frac{\lambda^2\sqrt{t}}{\sqrt{2\nu}}-\frac{|\Delta x|+|\Delta z|}{\sqrt{2\nu t}}\right).
\end{multline*}
This completes the proof of Theorem \ref{T:SecMom}.
\end{proof}

\section{Proof of Theorem \ref{T:LocB}}
\label{Sec:LocB}
\begin{proof}[Proof of Theorem \ref{T:LocB}]
% 
% {\vspace{1em}\noindent\bf Step 1.}
% In this step, we show that for some nonnegative function $h(y,v,t)$, the joint density of $(B_t,L_t^a)$ has the following form:
% \begin{align}\label{E_:LocB}
% f_{a,t}(y,v)= \left[\frac{|a|+|y|+v}{\sqrt{2\pi t^3}} \exp\left(-\frac{(|a|+|y|+v)^2}{2t}\right)
% + h(y,a,t)\delta_{0}(v)\right] \Indt{y\in\R}\Indt{v\ge 0}.
% \end{align}
% 
The case when $a=0$ is covered in \eqref{E:LocB0}.
Let $g:\R\times\R_+\mapsto \R$ be a smooth and bounded function. Denote the joint density
of $(B_t,L_t^a)$ by $f_{a,t}(x,y)$. Then
\[
\E\left[g(B_t,L_t^a)\right]
=\int_0^\infty \ud y \int_\R \ud x \:
g(x,y) f_{a,t}(x,y).
\]
First assume that $a>0$.
By the reflection principle (see \cite[p. 107]{RevuzYor99}), the density of $T_a:=\inf\left(s\ge 0\,, B_s=a\right)$ is
\[
f_{T_a}(s)= \frac{a}{\sqrt{2\pi s^3}}\exp\left(-\frac{a^2}{2s}\right),\qquad \text{for $s\ge 0$ and $a>0$.}
\]
Let $\{\calF_t\}_{t\ge 0}$ be the standard Brownian filtration. By the strong Markov property,
\begin{align*}
\E\left[g(B_t,L_t^a)\right] &=
\E\left[\E\left(g(B_t,L_t^a)|\calF_{T_a}\right) \Indt{T_a\le t}\right]
+\E\left[\E\left(g(B_t,L_t^a)|\calF_{T_a}\right) \Indt{T_a> t}\right]\\
&=\int_0^t \ud s \:
f_{T_a}(s) E\left[g(\hat{B}_{t-s}+a, \hat{L}_{t-s}^0)\right]
+\E\left(g(B_t, 0)\Indt{T_a>t}\right)\\
&=: I_1 + I_2,
\end{align*}
where $\hat{B}_t$ is a standard Brownian motion and $\hat{L}_t^0$ is its local time at the level $0$.
We first compute $I_1$. By \eqref{E:LocB0},
\begin{align*}
I_1 =& \iint_{\R\times\R_+}\ud y\ud v \: g(y+a,v)\int_0^t \ud s \frac{a}{\sqrt{2\pi s^3}}\exp\left(-\frac{a^2}{2s}\right)
\frac{|y|+v}{\sqrt{2\pi (t-s)^3}} \exp\left(-\frac{\left(|y|+v\right)^2}{2(t-s)}\right)\\
=&
\iint_{\R\times\R_+}\ud y\ud v \: g(y+a,v)
\frac{a+|y|+v}{\sqrt{2\pi t^3}} \exp\left(-\frac{(a+|y|+v)^2}{2t}\right),
\end{align*}
where in the last step, we have used the fact that the densities $f_{T_a}$ form a
convolution semigroup, namely, $f_{T_a}*f_{T_b}=f_{T_{a+b}}$ (see \cite[p.107]{RevuzYor99} or Lemma \ref{L:IntS} below for a short proof).

As for $I_2$, let $M_t=\sup_{s\le t} B_s$. Then, by the joint law of $(B_t,M_t)$ (see, e.g., \cite[Ex. 3.14 p. 110]{RevuzYor99}),
\begin{align*}
I_2 = \E\left[g(B_t,0)\Indt{M_t<a}\right]
&=\int_{0}^a \ud m \int_{-\infty}^m \ud y \: g(y,0)
\frac{\sqrt{2}(2m-y)}{\sqrt{\pi t^3}} \exp\left(-\frac{(2m-y)^2}{2t}\right)\\
&=\int_{-\infty}^a \ud y \: g(y,0) \int_{y_+}^a \ud m
\frac{\sqrt{2}(2m-y)}{\sqrt{\pi t^3}} \exp\left(-\frac{(2m-y)^2}{2t}\right)\\
&=\int_{-\infty}^a \ud y \: g(y,0)
\frac{1}{\sqrt{2\pi t}}\left(e^{-\frac{(2y_+-y)^2}{2t}} - e^{-\frac{(2a-y)^2}{2t}}\right),
\end{align*}
where $y_+=\max(y,0)$. Note that $2y_+-y=|y|$. 
Combining these two parts, we see that 
\begin{align*}
f_a(y,v)=\Bigg[\:&\frac{a+|y-a|+v}{(2\pi t)^3} \exp\left(-\frac{(a+|y-a|+v)^2}{2t}\right)\\
& + \frac{1}{\sqrt{2\pi t}}\left(e^{-\frac{y^2}{2t}}-e^{-\frac{(2a-y)^2}{2t}}\right)\Indt{y\le a} \delta_{0}(v)\Bigg] \Indt{v\ge 0}.
\end{align*}
If $a<0$, by symmetry, let $\tilde{B}_t=-B_t$ and $\tilde{L}_t^a$ be its local time.
Clearly, the density of $(\tilde{B}_t,\tilde{L}_t^{-a})$ is $f_{-a}(y,v)$.
Therefore, the law of $(B_t,L_t^a)$ is $f_{|a|}(-y,v)$.
Then use the fact that $|-y-|a|| =|y-a|$ and $|2|a|+y|=|2a-y|$.
This completes the proof of Theorem \ref{T:LocB}.
\end{proof}

\begin{lemma}\label{L:IntS}
 For $t>0$ and $a,b\ne 0$, the following integral is true
 \[
 \int_0^t |a b| (s(t-s))^{-3/2}\exp\left(-\frac{a^2}{2s}-\frac{b^2}{2(t-s)}\right) \ud s=
 \frac{|a|+|b|}{\sqrt{2 \pi t^3}}\exp\left(-\frac{(|a|+|b|)^2}{2t}\right).
 \]
\end{lemma}
\begin{proof}
Denote the integral by $I(t)$ and let $g_a(t)=|a|\: t^{-3/2}e^{a^2/(2t)}$.
By \cite[(28) on p. 146]{Erdelyi1954-I},
\begin{align}\label{E:Lap}
 \calL\left[g_a\right](z) = \sqrt{2\pi}e^{-\sqrt{2}|a|\sqrt{z}}\:.
\end{align}
So, $\calL[I](z) = \calL\left[g_a\right](z)\calL\left[g_b\right](z)
=2 \pi e^{-\sqrt{2}\left(|a|+|b|\right)\sqrt{z}}$.
Then use \eqref{E:Lap} for the inversion.
\end{proof}

\section{An alternative proof of Theorem \ref{T:SecMom}}
\label{Sec:FeynmanKac}
If the initial data $\mu$ is such that $\mu(\ud x)=u_0(x)\ud x$,
where $u_0$ is a bounded measurable function,
then we can use the Feynman-Kac representation \eqref{E:FKF} with $n=2$ and Theorem \ref{T:LocB}
to give an alternative proof of Theorem \ref{T:SecMom}.

\begin{proof}[An alternative proof of Theorem \ref{T:SecMom}]
For initial data specified above, $\E\left(u(t,x_1)u(t,x_2)\right)$
admits the Feynman-Kac representation \eqref{E:FKF}.
% Without loss of generality, we assume that $x_2\ge x_1$, so that $\Delta x=x_2-x_1\ge 0$.
Let
\[
W_{2\nu t}^1:=B_{\nu t}^1-B_{\nu t}^2\quad\text{and}\quad
W_{2\nu t}^2:=B_{\nu t}^1+B_{\nu t}^2.
\]
Note that $W_t^1$ and $W_t^2$ are two independent standard Brownian motions.
Then,
\begin{align*}
 \E\left(u(t,x_1)u(t,x_2)\right)&=
\E\Bigg(
u_0\left(x_1+2^{-1} [W_{2\nu t}^1+W_{2\nu t}^2]\right)
u_0\left(x_2+2^{-1} [W_{2\nu t}^2-W_{2\nu t}^1]\right)\\
&\qquad\qquad\times \exp\left(\lambda^2\int_0^t\delta_{x_2-x_1}\left(W_{2\nu s}^1\right)\ud s\right)\Bigg)\\
&=\int_\R \ud z \: G_{2\nu}(t,z)
\E\Bigg(
u_0\left(x_1+2^{-1} [W_{2\nu t}^1+z]\right)
u_0\left(x_2+2^{-1} [z-W_{2\nu t}^1]\right)\\
&\qquad\qquad\times
\exp\left(\lambda^2\int_0^t\delta_{x_2-x_1}\left(W_{2\nu s}^1\right)\ud s\right)\Bigg).
\end{align*}
Notice that
\[
\exp\left(\lambda^2\int_0^t\delta_{x_2-x_1}\left(W_{2\nu s}^1\right)\ud s\right)
=
\exp\left(\frac{\lambda^2}{2\nu}\int_0^{2\nu t}\delta_{x_2-x_1}\left(W_{s}^1\right)\ud s\right)
=\exp\left(\frac{\lambda^2}{2\nu} L_{2\nu t}^{x_2-x_1} \right),
\]
where $L_t^a$ be the local time of $W_t^1$ at the level $a$. Then
the expectation in the above integrand becomes
\begin{align}\label{E:YV}
\E\Bigg(
u_0\left(x_1+2^{-1} [W_{2\nu t}^1+z]\right)
u_0\left(x_2+2^{-1} [z-W_{2\nu t}^1]\right)
\exp\left(\frac{\lambda^2}{2\nu} L_{2\nu t}^{x_2-x_1} \right)\Bigg).
\end{align}
% In the following, we prove this theorem in two steps.
% 
% {\vspace{1em}\noindent\bf Step 1.}
% We first prove the case where $x_1=x_2=x$, which then give an alternative proof of \eqref{E:LocB0}.
% Let $f_t(y,v)$ be the joint law of $(W_t^1,L_t^0)$. Then 
% The expectation \eqref{E:YV} becomes
% \begin{align*}
% &\E\Bigg(
% u_0\left(x+2^{-1} [W_{2\nu t}^1+z]\right)
% u_0\left(x+2^{-1} [z-W_{2\nu t}^1]\right)
% \exp\left(\frac{\lambda^2}{2\nu} L_{2\nu t}^{0} \right)\Bigg)\\
% =&\int_\R\ud y \int_0^\infty\ud v \:
% u_0\left(x+2^{-1} [y+z]\right)
% u_0\left(x+2^{-1} [z-y]\right) f_{2\nu t} (y,v) e^{\frac{\lambda^2}{2\nu} v}
% \end{align*}
% By change of variables $z_1=(z+y)/2$ and $z_2=(z-y)/2$, we have that
% \begin{align*}
% \E\left(u(t,x)^2\right) = &
% \iint_{\R^2}\ud z_1 \ud z_2 \: u_0(x+z_1) u_0(x+z_2) 2 G_{2\nu}(t,z_1+z_2) \int_0^\infty\ud v\: f_{2\nu t}(z_1-z_2,v) e^{\frac{\lambda^2}{2\nu} v}.
% \end{align*}
% Notice that $2 G_{2\nu}(t,z_1+z_2)=G_{\nu/2}(t,(z_1+z_2)/2)$.
% Compare the above formula with \eqref{E:Sec} to see that
% \[
% f_{2\nu t}(y,v)=
% \frac{|y|+v}{\sqrt{2\pi (2\nu t)^3}} \exp\left(-\frac{\left(|y|+v\right)^2}{4\nu t}\right) \Indt{y\in\R} \Indt{v\ge 0}.
% \]
By Theorem \ref{T:LocB}, this expectation is equal to
\begin{align} \notag
\int_\R \ud y\:
& u_0\left(x_1+\frac{z+y}{2}\right)
u_0\left(x_2+\frac{z-y}{2}\right)
 \\ \label{E_:E}
\times& \int_0^\infty \ud v \:
\frac{|y-\Delta x|+|\Delta x|+v}{4\sqrt{\pi \nu^3  t^3}}\exp\left(-\frac{\left(|y-\Delta x|+v+|\Delta x| \right)^2}{4\nu t}+\frac{\lambda^2 v}{2\nu}\right)\\
+&
\int_{\sign(\Delta x) y\le |\Delta x|} \ud y\:
u_0(x_1+(z+y)/2)
u_0(x_2+(z-y)/2)
% h(y,|\Delta x|,2\nu t)
\left[G_{2\nu}(t,y)-G_{2\nu}(t,2\Delta x-y)\right].
\notag
\end{align}
By integration by parts, the $\ud v$-integration in \eqref{E_:E} is equal to
\begin{align*}
&\left.\frac{1}{\sqrt{4\pi\nu t}} e^{\frac{\lambda^2 v}{2\nu}}e^{-\frac{(|y-\Delta x|+ |\Delta x| +v)^2}{4\nu t}}
\right|_{v=\infty}^{v=0} + \frac{\lambda^2}{2\nu}\int_0^\infty \ud v \:
\frac{1}{\sqrt{4\pi \nu t}}e^{\frac{\lambda^2 v}{2\nu}}e^{-\frac{(|y-\Delta x|+ |\Delta x| +v)^2}{4\nu t}}\\
=&
G_{2\nu}(t,|y-\Delta x|+ |\Delta x| )
+ \frac{\lambda^2}{2\nu} e^{\frac{\lambda^4 t}{4\nu} -\frac{\lambda^2 (|y-\Delta x|+ |\Delta x| )}{2\nu}}
\Phi\left(\lambda^2\sqrt{\frac{t}{2\nu}}-\frac{|y-\Delta x|+ |\Delta x| }{\sqrt{2\nu t}}\right).
\end{align*}
Denote the $\ud y$-integral in \eqref{E_:E} by $I$.
By symmetry, we will only consider the case where $\Delta x>0$.
In this case, the $\ud y$-integral is from $-\infty$ to $\Delta x$.
By the change of variables $y'=2\Delta x -y$,  
\begin{equation}
\label{E_:II}
\begin{aligned}
\int_{-\infty}^{\Delta x} \ud y\:&
u_0(x_1+(z+y)/2)
u_0(x_2+(z-y)/2)
G_{2\nu}(t,2\Delta x-y)\\
=&\int_{\Delta x}^\infty\ud y'\:
u_0(x_1+(z+y')/2)
u_0(x_2+(z-y')/2)
G_{2\nu}(t,y').
\end{aligned}
\end{equation}
Hence,
\[
I=\int_{\R} \ud y\:
u_0(x_1+(z+y)/2)
u_0(x_2+(z-y)/2)
\left[
G_{2\nu}(t,y) \Indt{y\le \Delta x}-
G_{2\nu}(t,y) \Indt{y> \Delta x}
\right].
\]
Notice that
\begin{align*}
&\:\: G_{2\nu}(t,y) \Indt{y\le \Delta x}-
G_{2\nu}(t,y) \Indt{y> \Delta x} + G_{2\nu}(t,|y-\Delta x| + \Delta x)\\
=&
\left[G_{2\nu}(t,y) + G_{2\nu}(t,2\Delta x -y) \right]\Indt{y\le \Delta x}.
\end{align*}
Therefore, by \eqref{E_:II}, we know that
\begin{align*}
\int_{\R} \ud y\: &
u_0(x_1+(z+y)/2)
u_0(x_2+(z-y)/2) \\
&\times \Big(
 G_{2\nu}(t,y) \Indt{y\le \Delta x}-
G_{2\nu}(t,y) \Indt{y> \Delta x} + G_{2\nu}(t,|y-\Delta x| + \Delta x)\Big)
\\
&=
\int_{\R} \ud y\: 
u_0(x_1+(z+y)/2)
u_0(x_2+(z-y)/2) G_{2\nu}(t,y).
\end{align*}
Combining these calculations, we have that
\begin{align*}
\E\left(u(t,x_1)u(t,x_2)\right) = &
\iint_{\R^2}\ud z \ud y \: u_0(x_1+(z+y)/2) u_0(x_2+(z-y)/2) G_{2\nu}(t,z)\\
& \times
\left(G_{2\nu}(t,y)+\frac{\lambda^2}{2\nu} e^{\frac{\lambda^4 t}{4\nu} -\frac{\lambda^2 (|y-\Delta x|+ |\Delta x| )}{2\nu}}
\Phi\left(\lambda^2\sqrt{\frac{t}{2\nu}}-\frac{|y-\Delta x|+ |\Delta x| }{\sqrt{2\nu t}}\right)\right).
\end{align*}
Finally, by change of variables $z_1=x_1+(z+y)/2$ and $z_2=x_2+(z-y)/2$, we have that
\begin{align*}
\E\left(u(t,x_1)u(t,x_2)\right) = &
\iint_{\R^2}\ud z_1 \ud z_2 \: u_0(z_1) u_0(z_2) 2 G_{2\nu}(t,(z_1+z_2)- (x_1+x_2))\\
&\times
\Bigg(G_{2\nu}\left(t,\Delta z-\Delta x\right)
\\
&+ \frac{\lambda^2}{2\nu} \exp\left(\frac{\lambda^4 t}{4\nu} -\frac{\lambda^2 |z_1-z_2|+ |\Delta x|}{2\nu} \right)
\Phi\left(\lambda^2\sqrt{\frac{t}{2\nu}}-\frac{|z_1-z_2|+ |\Delta x| }{\sqrt{2\nu t}}\right)\Bigg).
\end{align*}
Notice that $2 G_{2\nu}(t,(z_1+z_2)- (x_1+x_2))=G_{\nu/2}(t,\bar{x}-\bar{z})$.
This proves \eqref{E:SecTP} (see \eqref{E:Sec_Inner2} and \eqref{E:K*}).
\end{proof}
\bigskip

\section{Feynman-Kac formula for measure-valued initial data}

\subsection{Initial data part (Proof of Theorem \ref{T:mun})} \label{SS:mun}
We need some lemmas. 
\begin{lemma}\label{L:p's G}
For any $t>0$, $s>0$, $x\in\R$, $y_i\in\R$, $i=1,\dots,p$, it holds that
\[
\int_\R G_1(s,x+z) \prod_{j=1}^p G_1(t,z-y_j)\ud z
\le
(p+1)^{p/2} \sqrt{\frac{t}{ps+t}}\: e^{\frac{px^2}{2(ps+t)}} \prod_{i=1}^p G_1((p+1)t,y_i).
% (2\pi)^{-p/2} t^{-(p-1)/2} (p s+t)^{-1/2} e^{\frac{px^2}{2(p s+t)}}
% \exp\left(-\frac{\sum_{j=1}^p y_j^2}{2(p+1)t}\right).
\]
\end{lemma}
\begin{proof}
Denote $\bar{y}=(y_{1}+\dots+y_{p})/p$.
Notice that
\begin{align*}
\prod_{j=1}^p G_1(t,z-y_{j})&
=(2\pi t)^{-p/2}\exp\left(-\frac{(z-\bar{y})^2}{2t/p}\right)
\exp\left(-\frac{y_{1}^2+\dots+y_{p}^2 -p\:\bar{y}^2}{2t}\right)
\\
&=(2\pi t)^{-(p-1)/2}p^{-1/2} G_1(t/p,z-\bar{y}) 
\exp\left(-\frac{y_{1}^2+\dots+y_{p}^2 -p\:\bar{y}^2}{2t}\right).
\end{align*}
Hence, by the semigroup property of the heat kernel, the $\ud z$ integral is equal to
\begin{align*}
\int_{\R} \ud z\: & G_1(s,x+z)\prod_{j=1}^p G_1(t,z-y_{j}) 
\\=& 
(2\pi t)^{-(p-1)/2}p^{-1/2} G_1(s+t/p,x+\bar{y})
\exp\left(-\frac{y_{1}^2+\dots+y_{p}^2 -p\: \bar{y}^2}{2t}\right)
\\
=&
(2\pi)^{-p/2} t^{-(p-1)/2} (p s+t)^{-1/2} 
\exp\left(-\frac{p(x+\bar{y})^2}{2(ps+t)}-\frac{y_{1}^2+\dots+y_{p}^2 -p\: \bar{y}^2}{2t}\right)\\
\le &
(2\pi)^{-p/2} t^{-(p-1)/2} (p s+t)^{-1/2} 
\exp\left(-\frac{p\bar{y}^2 - 2px^2}{4(ps+t)}-\frac{y_{1}^2+\dots+y_{p}^2 -p\: \bar{y}^2}{2t}\right)\\
= &
(2\pi)^{-p/2} t^{-(p-1)/2} (p s+t)^{-1/2} e^{\frac{px^2}{2(ps+t)}}
\exp\left(\frac{ p \bar{y}^2(t+2ps)}{4t(ps+t)}-\frac{y_{1}^2+\dots+y_{p}^2}{2t}\right)\\
\le &
(2\pi)^{-p/2} t^{-(p-1)/2} (p s+t)^{-1/2} e^{\frac{px^2}{2(ps+t)}}
\exp\left(\left(\frac{ t+2ps}{4t(ps+t)}-\frac{1}{2t}\right)(y_{1}^2+\dots+y_{p}^2)\right)\\
% \le &
% (2\pi)^{-p/2} t^{-(p-1)/2} (p s+t)^{-1/2} e^{\frac{px^2}{2(ps+t)}}
% \exp\left(\frac{ t+2ps}{4t(ps+t)}(y_{1}^2+\dots+y_{p}^2)\right)\\
\le &
(2\pi)^{-p/2} t^{-(p-1)/2} (p s+t)^{-1/2} e^{\frac{px^2}{2(ps+t)}}
\exp\left(-\frac{y_{1}^2+\dots+y_{p}^2}{2(p+1)t}\right)\\
=&
\sqrt{\frac{t}{ps+t}} (p+1)^{p/2} e^{\frac{px^2}{2(ps+t)}} \prod_{i=1}^p G_1((p+1)t,y_i).
\end{align*}
This proves Lemma \ref{L:p's G}.
\end{proof}

Let $T_t$ be the Ornstein Uhlenbeck semigroup of operators on $L^2(\Omega)$ with generator $L$.
\begin{lemma} \label{L:Ttn}
For any $\mu_i\in \calM_H(\R)$, $h_i\in H$, and $x_i\in\R$, $i=1,\dots,n$, it holds that
\[
T_t\left(\prod_{i=1}^n \mu_i\left(x_i+W_i(h_i)\right)\right)
=\prod_{i=1}^n \left[\mu_i*G_1(|h_i|^2(1-e^{-2t}),\cdot)\right](x_i+e^{-t} W_i(h_i)),
\]
for $t>0$,
where $W_i$ are i.i.d. zero mean Gaussian processes $W_i=\{W_i(h), h\in H\}$ with covariance function 
$\E(W_i(h)W_i(g))=\InPrd{h,g}$.
\end{lemma}
\begin{proof}
Fix $\epsilon>0$. Let $\mu_{i,\epsilon}(x) = (\mu_i*G_1(\epsilon,\cdot))(x)$.
By H\"older's inequality,
\[
\E\left[\prod_{i=1}^n\mu_{i,\epsilon}(W_i(h_i)+x_i)^2\right]
\le \prod_{i=1}^n \E\left[\mu_{i,\epsilon}(W_i(h_i)+x_i)^{2n}\right]^{1/n}.
\]
Notice that for all $p\ge 1$,
\begin{align*}
\E\left[|\mu_{i,\epsilon}(W_i(h_i)+x_i)|^{p}\right] 
 &=\int_\R \ud z\: G_1(|h_i|^2,x+z) \left|\int_\R G_1(\epsilon,z-y) \mu_i(\ud y)\right|^p\\
 &\le 
\int_{\R^p} |\mu_i|(\ud y_1)\dots |\mu_i|(\ud y_p) \int_\R \ud z
\: G_1(|h_i|^2,x+z) \prod_{j=1}^p G_1(\epsilon,z-y_j).
\end{align*}
By Lemma \ref{L:p's G} with $t=\epsilon$ and $s=|h_i|^2$,
\[
\E\left[|\mu_{i,\epsilon}(W_i(h_i)+x_i)|^{p}\right]\le
(p+1)^{p/2} e^{\frac{p x^2}{2(p|h_i|^2 +\epsilon)}} \left[(|\mu_i|*G_1(\epsilon(p+1),\cdot))(0)\right]^p,
\]
which is finite because $\mu_i\in\calM_H(\R)$.
Hence, $\prod_{i=1}^n\mu_{i,\epsilon}(W_i(h_i)+x_i)\in L^2(\Omega)$,
and one can apply Mehler's formula (see, e.g., \cite[Section 1.4.1]{Nualart06} to obtain that 
\begin{align}
\notag
T_t\left( \prod_{i=1}^n \mu_\epsilon \left(x_i+W_i(h_i)\right)\right)
&=\E'\left[
\prod_{i=1}^p
\mu_{i,\epsilon}\left(x_i+e^{-t}W_i(h_i) + \sqrt{1-e^{-2t}}\: W_i'(h_i)\right)
\right]
\\
\notag
&=
\int_{\R^n} \prod_{i=1}^n \mu_{i,\epsilon}\left(x_i+e^{-t}W_i(h_i) +y_i\right)  G_1\left(|h_i|^2(1-e^{-2t}),y_i\right)\ud y_i
\\
\notag
&=\prod_{i=1}^n\left[\mu_{i,\epsilon} *G_1(|h_i|^2(1-e^{-2t}),\cdot)\right](x+e^{-t} W_i(h_i))
\\
&=
\prod_{i=1}^n\left[\mu_i *G_1(\epsilon+|h_i|^2(1-e^{-2t}),\cdot)\right](x_i+e^{-t} W_i(h_i)).
\label{E:Step2Tt}
\end{align}
% 
% {\bigskip\noindent\bf Step 3.~} 
Finally, $T_t\left(\prod_{i=1}^n \mu_i\left(x_i+W_i(h_i)\right)\right)$ is the $L^2(\Omega)$-limit
of $T_t\left( \prod_{i=1}^n\mu_{i,\epsilon} \left(x_i+W_i(h_i)\right)\right)$,
and this limit can be obtained by sending $\epsilon$ to zero in \eqref{E:Step2Tt}.
This proves Lemma \ref{L:Ttn}.
\end{proof}
\bigskip
\begin{proof}[Proof of Theorem \ref{T:mun}]
Without loss of generality, assume that $\mu_i\ge 0$. Following \cite{NualartVives92}, we write $(I-L)^{-\alpha/2}$ in the following form
\[
(I-L)^{-\alpha/2} = \Gamma(\alpha/2)^{-1} \int_0^\infty e^{-t}t^{\alpha/2-1}T_t \ud t.
\]
By Lemma \ref{L:Ttn},
\begin{align}
\notag
&\Norm{(I-L)^{-\alpha/2} \prod_{i=1}^n \mu(W_i(h_i)+x_i)}_p\\
\notag
&= \Norm{\Gamma(\alpha/2)^{-1}\int_0^\infty e^{-t} t^{\alpha/2-1}
\prod_{i=1}^n \left[\mu_i*G_1(|h_i|^2(1-e^{-2t}),\cdot)\right](e^{-t}W_i(h_i)+x_i))\ud t }_p\\
&\le \Gamma(\alpha/2)^{-1}
\int_0^\infty e^{-t} t^{\alpha/2-1}\Norm{\prod_{i=1}^n \left[\mu_i*G_1(|h_i|^2(1-e^{-2t}),\cdot)\right](e^{-t}W_i(h_i)+x_i))}_p\ud t.
\label{E:N>I-L}
\end{align}
Now
\begin{align*}
&\Norm{\prod_{i=1}^n\left[\mu_i*G_1(|h_i|^2(1-e^{-2t}),\cdot)\right](e^{-t}W_i(h_i)+x_i))}_p^p 
 \\
=&\E\left(\left|\prod_{i=1}^n \int_\R \frac{e^{-\frac{(x_i+e^{-t}W_i(h_i)-y)^2}{2|h_i|^2(1-e^{2t})}}}{\sqrt{2\pi |h_i|^2(1-e^{-2t})}} \mu_i(\ud y)\right|^p\right)
\\
=&
\int_{\R^n} \ud z_1\dots\ud z_n \left(\prod_{i=1}^n G_1(S_i,x_i+z_i)\right)
\int_{\R^{np}} \prod_{j=1}^p \prod_{i=1}^n G_1(T_i,z_i-y_{ij}) \mu_i(\ud y_{ij})\\
=&
\int_{\R^{np}} \left(\prod_{j=1}^p \prod_{i=1}^n \mu_i(\ud y_{ij})\right)
\prod_{i=1}^n \int_{\R} \ud z_i \: G_1(S_i,x_i+z_i)
\prod_{j=1}^p G_1(T_i,z_i-y_{ij})
\end{align*}
where
\[
T_i=|h_i|^2(1-e^{-2t}) \quad\text{and}\quad 
S_i= e^{-2t} |h_i|^2.
\]
% For $i\in \{1,\dots,n\}$, denote $\bar{y}_i=(y_{i1}+\dots+y_{ip})/p$.
By Lemma \ref{L:p's G}, the $\ud z_i$ integral is bounded by
\begin{align*}
\int_{\R} \ud z_i\:  G_1(S_i,x_i+z_i)\prod_{j=1}^p G_1(T_i,z_i-y_{ij})
\le &
(p+1)^{p/2}\sqrt{\frac{T_i}{pS_i+T_i}} \: e^{\frac{px_i^2}{2(pS_i+T_i)}} \prod_{j=1}^p G_1((p+1)T_i,y_{ij})
% \\
% \le &
% (p+1)^{p/2}e^{\frac{px_i^2}{2(pS_i+T_i)}} \frac{|h_i|^p}{T_i^{p/2}} \prod_{j=1}^p G_1((p+1)|h_i|^2,y_{ij}).
% \\
% \le &
% (2\pi)^{-p/2} T_i^{-(p-1)/2} (p S_i+T_i)^{-1/2} e^{\frac{px_i^2}{2(pS_i+T_i)}}
% \exp\left(-\frac{y_{i1}^2+\dots+y_{ip}^2}{2T_i}\right)\\
% \le &
% (2\pi)^{-p/2} T_i^{-(p-1)/2} (p S_i+T_i)^{-1/2} e^{\frac{px_i^2}{2(pS_i+T_i)}}
% \exp\left(-\frac{y_{i1}^2+\dots+y_{ip}^2}{2|h_i|^2}\right)\\
\\
\le&T_i^{(1-p)/2} \frac{|h_i|^{p}}{\sqrt{pS_i+T_i}}e^{\frac{px_i^2}{2(pS_i+T_i)}}
\prod_{j=1}^p G_1(|h_i|^2(p+1), y_{ij}),
\end{align*}
where we have applied the inequality
\[
G_1((p+1)T_i,y_{ij})\le \frac{|h_i|}{\sqrt{T_i}}G_1((p+1)|h_i|^2,y_{ij}).
\]
Hence, 
\[
 \Norm{\prod_{i=1}^n\left[\mu_i*G_1(|h_i|^2(1-e^{-2t}),\cdot)\right](e^{-t}W_i(h_i)+x_i))}_p 
 \le 
 \prod_{i=1}^n
 \frac{T_i^{\frac{1}{2p}-\frac{1}{2}}|h_i|e^{\frac{x_i^2}{2(pS_i+T_i)}}}{(pS_i+T_i)^{1/(2p)} } J_0((p+1)|h_i|^2,0).
\]
By substituting the above upper bound into \eqref{E:N>I-L},
we see that the integral in \eqref{E:N>I-L} converges provided that
\[
\int_{0_+}t^{\frac{\alpha}{2}-1+\frac{n}{2p}-\frac{n}{2}} \ud t <\infty,
\]
where we have used the fact that  
\[
 pS_i+T_i\ge |h_i|^2 p e^{-2t}\quad\text{and}\quad T_i=2|h_i|^2 (t + O(t^2)).
\]
Therefore, $\alpha+n/p>n$.
This completes the proof of Lemma \ref{T:mun}.
\end{proof}

\subsection{Local time part (Proof of Theorem \ref{T:fL})}\label{SS:fL}

\begin{proof}[Proof of Theorem \ref{T:fL}]
This is a slight extension of \cite[Theorem 1]{AiraultRenZhang00}.
The proof consists three steps.
Let $L^x_t$ be the local time of the standard one-dimensional Brownian motion. 
 
{\bigskip\noindent\bf Step 1.~}
Fix $\epsilon>0$. 
Denote 
\[
F_{\epsilon,x}(y)=
\begin{cases}
 1& \text{if $y>x+\epsilon$,}\\
 (y-x+\epsilon)/(2\epsilon) & \text{if $|y-x|\le \epsilon$,}\\
 0 &\text{if $y<x-\epsilon$.}
\end{cases}
\]
Define
\[
N_\epsilon(x,t)= \int_0^t F_{\epsilon,x}(B_s)\ud B_s\quad\text{and}\quad
N(x,t)=\int_0^t \Indt{B_s>x}\ud B_s.
\]
By Tanaka's formula, 
\[
L_t^x= (B_t-x)^+ -(-x)^+ -N(x,t).
\]
Notice that $L(x,t)\in [0,t]$ is a bounded random variable.
On the other hand, let 
\[
L_{\epsilon,t}^x :=(B_t-x)^+ -(-x)^+ -N_\epsilon(x,t).
\]
Notice that
\[
N(x+\epsilon,t)\le N_\epsilon(x,t)\le N(x-\epsilon,t),
\]
Using the fact that for all $x\in\R$ and $y\ge 0$
\[
x^+\le (x+y)^+\le x^++y\quad\text{and}\quad
x^+\ge (x-y)^+\ge x^+-y,
\]
we see that
\begin{align*}
L_{\epsilon,t}^x
&\le (B_t-(x+\epsilon) + \epsilon)^+-(-(x+\epsilon) +\epsilon )^+ -N(x+\epsilon,t)\le L_t^{x+\epsilon} + \epsilon,
\end{align*}
and 
\begin{align*}
L_{\epsilon,t}^x
&\ge (B_t-(x-\epsilon) - \epsilon)^+-(-(x-\epsilon) -\epsilon )^+ -N(x-\epsilon,t)\ge L_t^{x- \epsilon} - \epsilon.
\end{align*}
Hence, it holds that
\begin{align}\label{E:Leps}
L_t^{x- \epsilon} - \epsilon \le L_{\epsilon,t}^x \le
L_t^{x+\epsilon} + \epsilon.
\end{align}

{\bigskip\noindent\bf Step 2.~}
Now assume that $\epsilon\in (-1,1)$.
By telescoping sum,
\begin{align}\notag
\left|f(z_1,\dots,z_m)-f(z_1',\dots,z_m')\right|\le &\quad
  \left|f(z_1,z_2,\dots,z_m)-f(z_1',z_2,\dots,z_m)\right| \\
  \notag
&+\left|f(z_1',z_2,z_3,\dots,z_m)-f(z_1',z_2',z_3,\dots,z_m)\right| \\
\notag
&+\dots\\
\label{E:Tele}
&+\left|f(z_1',\dots,z_{n-1}',z_m)-f(z_1',\dots,z_{n-1}',z_m')\right|.
\end{align}
Because $\frac{\partial^2}{\partial x_1^2} f\ge 0$, the first term to the right of \eqref{E:Tele}
satisfies that
\[
\left|f(z_1,z_2,\dots,z_m)-f(z_1',z_2,\dots,z_m)\right|
\le \left(|f_1(z_1,z_2,\dots,z_m)|+|f_1(z_1',z_2,\dots,z_m)|\right) \: |z_1-z_1'|.
\]
Now replace $z_i$ and $z_i'$ by $L^i:=L^{i, x_i}_t $ and $L^i_\epsilon:=L^{i, x_i}_ {\epsilon,t}$, $i=1,\dots, m$, respectively.
Denote the quantity in \eqref{E:Ct} by
\[
C_{p}:=C_p(t,x_1,\dots,x_m).
\]
Let $1/r+1/q=1$, $q\ge 2$. By H\"older's inequality,
the first term on the right hand side of \eqref{E:Tele} satisfies that
\begin{align*}
&\hspace{-3em}\Norm{f(L^1,L^2,\dots,L^m)-f(L^1_\epsilon,L^2,\dots,L^m) }_p\\
&\le \Norm{ |f_1(L^1,L^2,\dots,L^m)|+|f_1(L^1_\epsilon,L^2,\dots,L^m)| }_{pq} \: \Norm{L^{1,x_1}_t-L^{1,x_1}_{\epsilon ,t}  }_{pr}\\
&\le 2\sup_{\epsilon_1\in (-1,1)}\Norm{ f_1(L^{1,x_1+\epsilon_1}_t +\epsilon_1,L^2,\dots,L^m) }_{pq} \: \Norm{L^{1,x_1}_t -L_{\epsilon,t} ^{1,x_1}  } _{pr}\\
&\le 2\:  C_{pq} \Norm{L^{x_1}_t -L_{\epsilon,t} ^{x_1} }_{pr},
\end{align*}
where we have used the fact \eqref{E:Leps}.
This inequality is true for all the $n$ terms on the right hand side of \eqref{E:Tele}
under the same replacements.
Therefore, 
\begin{align}\label{E:(2.8)}
\Norm{f(L^1,\dots,L^m)-f(L^1_\epsilon,\dots,L^m_\epsilon) }_p
\le 2\: C_{pq} \sum_{i=1}^m  \Norm{L^{x_i}_t -L^{x_i} _{\epsilon,t}  }_{pr}
\le 2 m \widetilde{C}_{pq} \epsilon^{1/2},
\end{align}
where the last inequality is due to (2.8) of \cite{AiraultRenZhang00}.

By H\"older's inequality with $1/r+1/q=1$, we see that
\begin{align*}
\E\left[\Norm{Df(L_\epsilon^1,\dots,L_\epsilon^m)}_H^p\right]&
\le 2^{p-1}
\sum_{i=1}^m
\E\left(|f_i(L_\epsilon^1,\dots,L_\epsilon^m)|^p \Norm{D L_{\epsilon,t}^{i,x}  }_H^p\right)\\
&\le
2^{p-1}
\sum_{i=1}^m
\E\left(|f_i(L_\epsilon^1,\dots,L_\epsilon^m)|^{pq}\right)^{1/q} \E\left(\Norm{D L_{\epsilon,t}^{i,x} }_H^{pr}\right)^{1/r}\\
&\le 
2^{p-1}
\:C_{pq}^p \sum_{i=1}^m\Norm{D  L_{\epsilon,t}^{i,x}  }_{L^{pr}(\Omega;H)}^{p},
\end{align*}
where we have applied \eqref{E:Leps} and \eqref{E:Ct} as before.
Hence, by (2.9) of \cite{AiraultRenZhang00},
\begin{align}\notag
\Norm{Df(L_\epsilon^1,\dots,L_\epsilon^m}_{L^{p}(\Omega;H)}
&\le 2 \: m C_{pq}
\sum_{i=1}^m
\left(\Norm{D (B_t-x_i)^+}_{L^{pr}(\Omega;H)}+\Norm{D N_\epsilon(x_i,t)}_{L^{pr}(\Omega;H)}\right)\\
&\le
2 \: \widetilde{C}_{pq} \: \epsilon^{-1+\frac{1}{2q'}}\quad\text{for all $q'>1$.}
\label{E:(2.9)}
\end{align}
Therefore, by the same arguments as the proof of Theorem 1 in \cite{AiraultRenZhang00},
we see that \eqref{E:(2.8)} and \eqref{E:(2.9)} imply that
$f(L^1,\dots,L^m)\in \D^{\alpha,p}(\R)$ for all $p>1$ and $\alpha<1/2$.

{\bigskip\noindent\bf Step 3.~}
In this final step, we need to verify that $f(z_1,\dots,z_m)=\exp\left(\lambda^2 \sum_{j=1}^m z_j\right)$
satisfies the condition \eqref{E:Ct}.
By H\"older's inequality, we have that, for all $\epsilon_j\in(-1,1)$,
\begin{align*}
\Norm{f_i(L^{1,x_1+\epsilon_1} _t +\epsilon_1,\dots,L^{m,x_m+\epsilon_m} _t+\epsilon_m)}_p&=
\lambda^2\Norm{\exp\left(\lambda^2 \sum_{j=1}^m [L^{j,x_j+\epsilon_j} _t+\epsilon_j]\right)}_p\\
&\le
\lambda^2\prod_{j=1}^m
\Norm{\exp\left(\lambda^2 [L^{j,x_j+\epsilon_j} _t+\epsilon_j]\right)}_{mp}\\
&\le
\lambda^2 e^{\lambda^2 m}\prod_{j=1}^m
\Norm{\exp\left(\lambda^2 L^{j,x_j+\epsilon_j} _t\right)}_{mp}.
\end{align*}
By Corollary \ref{C:ExpL},
\[
\Norm{\exp\left(\lambda^2 L^{ x_j+\epsilon_j} _t  \right)}_{mp} =
\E\left[\exp\left(\lambda^2 mp L^{ x_j+\epsilon_j} _t \right) \right]^{\frac{1}{mp}}
\le \left[2e^{\lambda^4 m^2p^2 t/2}+1\right]^{\frac{1}{mp}}.
\]
Therefore,
\begin{multline*}
\max_{1\le i\le m}\mathop{\sup_{\epsilon_j \in(-1,1)}}_{1\le j\le m}\Norm{f_i(L^{ 1,x_1+\epsilon_1} _t +\epsilon_1,\dots,L^{ m,x_m+\epsilon_m} _t+\epsilon_m)}_p
\le \lambda^2 e^{ \lambda^2 m}
\left[2e^{\lambda^4 m^2p^2 t/2}+1\right]^{\frac{1}{p}}.
% \sup_{\epsilon_1,\dots,\epsilon_n \in(-1,1)}
% \prod_{j=1}^n
% \Norm{\exp\left(\lambda^2 L(x_j+\epsilon_j,t)\right)}_{np}.
\end{multline*}
This completes the proof of Theorem \ref{T:fL}.
\end{proof}

\addcontentsline{toc}{section}{Bibliography}
\def\polhk#1{\setbox0=\hbox{#1}{\ooalign{\hidewidth
  \lower1.5ex\hbox{`}\hidewidth\crcr\unhbox0}}} \def\cprime{$'$}

\end{document}